\documentclass[11pt]{amsart}
\usepackage{amsmath,amssymb,amsfonts}%
\usepackage{amsthm}%
\usepackage[title]{appendix}%
\usepackage{xcolor}%
\usepackage{comment}
\usepackage{ytableau}
\newtheorem{thm}{Theorem}[section]
\newtheorem{definition}[thm]{Definition}
\newtheorem{lemma}[thm]{Lemma}
\newtheorem{corollary}[thm]{Corollary}
\newtheorem{prop}[thm]{Proposition}
\newtheorem{example}[thm]{Example}

\newtheorem{question}[thm]{Question}
\newtheorem{remark}[thm]{Remark}

\title{A new approach to the grammic monoid}

\author[M. Johnson]{Marianne Johnson}
\address{University of Manchester}
\email{Marianne.Johnson@manchester.ac.uk}

\author[A. Malheiro]{Ant\'{o}nio Malheiro}
\address{Universidade NOVA de Lisboa}
\email{ajm@fct.unl.pt}

\begin{document}
\begin{abstract}We give an alternative description of the grammic monoid in terms of weakly increasing subsequences. Specifically, we show that words $u,v$ in the generators $\{1,\ldots, n\}$ determine the same element of the grammic monoid of rank $n$ if and only if for all $1 \leq p \leq q$, the maximum length of a weakly increasing subsequence on alphabet $\{p,\ldots, q\}$ is the same in $u$ and $v$. Our proof makes use of a particular tropical representation of the plactic monoid determined by such sequences: we demonstrate that the grammic monoid is isomorphic to the image of this representation, and (by applying a result of the first author and Kambites) immediately deduce that the grammic monoid of rank $n$ satisfies exactly the same semigroup identities as the monoid of $n \times n$ upper triangular tropical matrices. This gives a partial generalisation of a result of Volkov, who has shown that the grammic monoid of rank $3$ satisfies exactly the same semigroup identities as the plactic monoid of rank $3$ which in turn is known (by applying a result of the first author and Kambites) to satisfy the exactly the same semigroup identities as the monoid of $3 \times 3$ upper triangular tropical matrices. Furthermore, we find that the grammic monoid of infinite rank does not satisfy any non-trivial semigroup identity, and demonstrate that the grammic congruence satisfies some useful compatibility properties.
\end{abstract}

\maketitle
\section{Introduction}
The grammic monoid of rank $n$ was introduced by Choffrut \cite{C} and is defined in terms of an action of the free monoid over $[n]:=\{1, \ldots, n\}$ on
the set of rows of semistandard Young tableaux with entries from $[n]$ determined by considering the effect of Schensted insertion on bottom rows only. Choffrut's definition was inspired by the \emph{stylic monoids} of Abram and Reutenauer \cite{AR}, which arise from an action of the free monoid on the set of columns defined by Schensted insertion. These two families of monoids are very different: for example, since the number of columns is finite,  the stylic monoid is finite, whereas the grammic monoid is infinite. Following the publication of \cite{AR}, the stylic monoid has attracted attention from a number of authors who have described:  the semigroup identities satisfied by the stylic monoids \cite{Vs} (see also \cite{AR3} for an alternative proof of the same result obtained independently and contemporaneously by a different method), the quiver algebras of stylic monoids \cite{AR2}, and the extended Green's relations of the stylic monoid \cite{GR}. By contrast, relatively little is known about the structure and properties of the grammic monoids, although it is clear that grammic monoids are closely related to the well-studied \emph{plactic monoids} \cite{S, Kn, LS}; specifically, the grammic monoid of rank $n$ is a quotient of the plactic monoid of rank $n$.

The grammic monoids of ranks $1$, $2$ and $3$ are well understood: the grammic monoid of rank $1$ is nothing more than the free monoid of rank $1$; the grammic monoid of rank $2$ coincides with the plactic monoid of rank $2$; and  Choffrut \cite{C} has shown that the grammic monoid of rank $3$ is the (proper) quotient of plactic monoid of rank $3$ by the congruence generated by the pair $(bacb,cbab)$ where $a<b<c$. It follows from a result of Daviaud, the first author and Kambites \cite{DJK} (see also \cite[Remark 4.6]{JK}) that the grammic monoid of rank $2$ satisfies \emph{exactly the same} semigroup identities as the monoid of $2 \times 2$ upper triangular tropical matrices. Volkov \cite{V} has shown that the grammic monoid of rank $3$ satisfies exactly the same semigroup identities as the plactic monoid of rank $3$, which by a result of the first author and Kambites \cite{JK} yields that the grammic monoid of rank $3$ satisfies exactly the same semigroup identities as the monoid of $3 \times 3$ upper triangular tropical matrices.

The aim of this article is to extend the results of \cite{C} and \cite{V} to the rank $n$ case, by giving a new description of the grammic monoid (Theorem \ref{thm:sequences}) that allows us to demonstrate (in Corollary \ref{cor:identities}) that the variety $\mathbf{V}_n$ of semigroups generated by the $n \times n$ upper triangular tropical matrices coincides with the variety of semigroups generated by the grammic monoid of rank $n$. The variety $\mathbf{V}_n$ has received substantial attention from a number of authors \cite{A,CKKMO,CHLS, DJK, I, JK, K, O}. Moreover, Corollary \ref{cor:identities} makes an interesting contrast with the results of \cite{AR3}, where it was shown that the stylic monoid of rank $n$ satisfies exactly the same identities as the monoid of $(n+1) \times (n+1)$ \emph{uni-triangular} tropical matrices, and hence generates the variety $\mathbf{J}_n$. As remarked in \cite{V}, the variety $\mathbf{J}_n$ is generated by several interesting (and seemingly disparate) monoids arising from different parts of mathematics. The results of \cite{DJK}, \cite{JK} and \cite{V} gave the first clues that a similar clustering phenomenon may be true of the variety $\mathbf{V}_n$.  The tantalising question (first raised in both \cite{CKKMO} and \cite{I2}) of whether the variety of semigroups generated by the monoid of $n \times n$ upper triangular matrices is equal to the variety generated by the plactic monoid of rank $n$ remains open.

The paper is structured as follows: in Section 2 we outline the preliminary definitions and results needed from \cite{C}; in Section 3 we prove our main result exhibiting an alternative description of the grammic monoid using an existing (non-faithful) tropical matrix representation of the plactic monoid; in Section 4 we demonstrate that we may immediately deduce several results concerning the identities satisfied by grammic monoids, in particular giving a generalisation of the main result of \cite{V}; in Section 5 we prove that the grammic monoid is compatible with restriction to alphabet intervals, packing, standardisation and Sch\"utzenberger involution; in Section 6 we turn our attention to questions concerning presentations of the grammic monoid exhibiting some generalisations of the Knuth relations and demonstrating that grammic monoids are quotients of the so-called `patience sorting monoids'; and in Section 7 we outline several open questions for future research.

\section{The grammic monoid}

\subsection{Words, tableaux and the plactic monoid}
For a finite alphabet $\Sigma$, we write $\Sigma^*$ to denote the free monoid on $\Sigma$, that is, the set of all finite
(possibly empty) words over $\Sigma$ under the operation of concatenation, and we write $\varepsilon$ to denote the empty word. We write $\Sigma^+$ for the
subsemigroup of non-empty words in $\Sigma^*$, which is the free semigroup on $\Sigma$. For $w \in \Sigma^*$ and
and $a \in \Sigma$ we write $|w|$ for the length of $w$ and $|w|_a$ for the number of occurrences of the
letter $a$ in $w$ (both of which are zero if $w = \varepsilon$). For $u, w \in \Sigma^+$ with $u=u_1 \cdots u_d$ where each $u_i \in \Sigma$, we say that $u$ is a scattered subword of $w$ if there exist words $v_0, v_1, \ldots, v_{d} \in \Sigma^*$ such that $w=v_0u_1v_1\cdots u_dv_d$. 

A partition of a natural number $k$ is any tuple of the form $(\lambda_1, \ldots, \lambda_d)$ where $\lambda_1 \geq  \cdots \geq \lambda_d >0$ and $\sum_{i=1}^d \lambda_i = k$. Partitions may be visually represented by Young diagrams: the Young diagram corresponding to a partition $\lambda=(\lambda_1,\ldots, \lambda_d)$ is the left-aligned array of identically-sized boxes, with  $\lambda_i$ boxes in the $i$th row, and where rows are indexed from bottom to top (i.e. there are $\lambda_1$ boxes in the bottom row). A semistandard Young tableau over an ordered alphabet $\mathcal{A}$ is a filling of a Young diagram with exactly one entry from $\mathcal{A}$ placed in each box of the diagram in such a way that reading from left to right along each row yields a weakly increasing sequence, and reading from bottom to top along any column yields a
strictly increasing sequence. An
example of a semistandard Young tableau with entries taken from the alphabet $\mathcal{A} = \{1<\cdots<9\}$ is given below:
\begin{center}
\begin{ytableau}
	7 & 8\\
	4 & 6\\
	3 & 5 & 7 \\
	2 & 2 & 5 &6
\end{ytableau}
\end{center}
The terminology `semistandard Young tableau' is used to distinguish from related concepts; since none such appear in this article, to abbreviate the text we will often write simply `tableau' to mean `semistandard Young tableau'.  It will also be convenient to have the notion of an empty tableau (i.e. containing no boxes, and hence also no entries), which we shall denote by $\emptyset$.

The following algorithm takes a tableau $T$ and a symbol $a \in \mathcal{A}$ and produces a new tableau by `inserting' $a$ into $T$ as follows:

\medskip

\noindent\textbf{Schensted's row insertion algorithm}
\\
\textit{Input:} A tableau $T$ and a symbol $a \in \mathcal{A}$.\\ \
\textit{Output:} A tableau $T \leftarrow a$.
\begin{enumerate}
\item[0.] If $T$ is the empty tableau, output the tableau consisting of a single box containing $a$.
\item If $a$ is greater than or equal to every entry in the bottom row of $T$, add $a$ as an entry at the rightmost end of
  $T$ and output the resulting tableau.

\item Otherwise, let $z$ be the leftmost entry in the bottom row of $T$ that is strictly greater than $a$. Replace $z$ by
  $a$ in the bottom row and recursively insert $z$ into the tableau formed by removing this new bottom row from $T$. (Note
  that the recursion may end with an insertion into an `empty row' above all the existing rows of the input tableau $T$.)

\end{enumerate}

\medskip
In step 2 of the above algorithm, we say that insertion of $a$ `bumps' entry $z$ into row $2$. The entry $z$ may in turn `bump' an entry from row  $2$ into row $3$, and so on.

To each word $w \in \mathcal{A}^*$ one can associate a tableau which we denote by $P(w)$ by repeated application of Schensted's row insertion algorithm; specifically, we associate the empty tableau to the empty word and if $w=w_1 \cdots w_k \in \mathcal{A}^+$ where each $w_i \in \mathcal{A}$, then $P(w) = (\cdots ((\emptyset \leftarrow w_1) \leftarrow w_2) \cdots \leftarrow w_k$.  In general, many words produce the same tableau; given a tableau $T$ one can construct a word $w$ for which $P(w)=T$ by either reading the rows left-to-right and top-to-bottom (the `row reading') or reading the columns top-to-bottom and left-to-right (the `column reading'). For example,  $P(78463572256)$ is equal to the tableau pictured above, as is $P(74328652756)$. 

The mapping $w \mapsto P(w)$
induces an equivalence relation $\equiv_{\rm plac}$ on $\mathcal{A}^*$ defined by $w \equiv_{\rm plac} v$ if and only if $P(w)=P(v)$, and this relation is seen to be a (two-sided) congruence on $\mathcal{A}^*$. Alternatively, $\equiv_{\rm plac}$ can be defined as the least congruence on $\mathcal{A}^*$ generated by the set of Knuth relations over $\mathcal{A}$:
$$\mathcal{R}_{{\rm plac}, \mathcal{A}}:=\{(yz x,yxz):  x,y, z \in \mathcal{A}, x < y \leq z \} \cup \{(zxy,xzy): x,y, z  \in \mathcal{A}, x \leq  y < z \}.$$

If $\mathcal{A}$ and $\mathcal{A}'$ are ordered alphabets of the same cardinality, then it is clear that the quotients by the corresponding congruences are isomorphic. Throughout we shall write $[n]$ to denote the ordered alphabet $ \{1<2 < \cdots < n\}$, $\mathbb{N}$ to denote the set of positive integers and $\mathbb{N}_0$ to denote the set of non-negative integers. The plactic monoid of rank $n$ is then the quotient $\mathbb{P}_n:=[n]^*/  \equiv_{\rm plac}$, and the plactic monoid of countably infinite rank is $\mathbb{P}_{\mathbb{N}}:=\mathbb{N}^*/\equiv_{\rm plac}$.

\subsection{The grammic monoid}
We view the elements of $\mathbb{N}_0^n$ as possible rows of a tableau over alphabet $[n]$; specifically $\gamma = (\gamma_1, \ldots, \gamma_n) \in \mathbb{N}_0^n$ corresponds to the row with total number of $i$'s equal to $\gamma_i$ for all $i$. For $\gamma \in \mathbb{N}_0^n$ we define a partial map
$$f_\gamma: [n] \rightarrow [n], \quad  f_\gamma(a) = {\rm min} \{t: t>a \mbox{ and } \gamma_t>0\} \mbox{ if this set is non-empty}.$$
Note that where $f_\gamma(a)$ is defined, it is equal to the first element to be `bumped' during insertion of $a$ into any tableau with  bottom row $\gamma$, and in this case $a<f_\gamma(a)$. If no bumping occurs during the insertion of $a$, then $f_\gamma(a)$ is not defined.  Now let $e_{\gamma,a} \in \mathbb{Z}^n$ be the vector with $i$th entry for $1 \leq i \leq n$ defined as follows
$$(e_{\gamma,a})_i=\begin{cases} 1 & \mbox{ if } i=a\\
-1 & \mbox{ if } i=f_\gamma(a)\\
0 & \mbox{ else}.
\end{cases}$$
Now, each  $a \in [n]$ acts on $\mathbb{N}_0^n$ via $\gamma \cdot a = \gamma + e_{\gamma,a}$ corresponding to the effect (on the bottom row only) of inserting symbol $a$ into any tableau with bottom row $\gamma$ using Schensted's algorithm. This extends recursively to give an action of $[n]^*$ on $\mathbb{N}_0^n$.

The action of words on rows gives rise to a congruence on $[n]^*$ defined by $w \equiv_{\rm gram} v$ if $\gamma \cdot w = \gamma \cdot v$ for all $\gamma \in \mathbb{N}_0^n$. The grammic monoid (of rank $n$) is then the quotient $\mathbb{G}_n:=[n]^*/\equiv_{\rm gram}$. We may also define a congruence $\equiv_{\rm Gram}$ on $\mathbb{N}^*$ by $w \equiv_{\rm Gram} v$ if $\gamma \cdot w = \gamma \cdot v$ for all (finite) bottom rows $\gamma$ over alphabet $\mathbb{N}$. The grammic monoid of infinite rank is then $\mathbb{G}_{\mathbb{N}}:=\mathbb{N}^*/\equiv_{\rm Gram}$.
We recall several key facts from \cite{C}:
\begin{itemize}
\item If $P(u) = P(v)$ (that is, if $u$ and $v$ are equivalent in the plactic monoid), then $u \equiv_{\rm gram} v$, but for $n \geq 3$ the converse does not hold.
\item  By taking $\gamma = (N, \ldots, N)$ where $N >|u|, |v|$, one can deduce that if $u \equiv_{\rm gram} v$, then $|u|_a=|v|_a$ for all $a \in [n]$, and hence in particular $|u|=|v|$ (see \cite[Lemma 1]{C}).
\item By taking $\gamma = (0, \ldots, 0)$, one finds that if $u \equiv_{\rm gram} v$, then $P(u)$ and $P(v)$ must have the same bottom row (see  \cite[Lemma 2]{C}).
\item For $u$ and $v$ of the same length $k$, in order to determine whether $u \equiv_{\rm gram} v$ it suffices to consider the action of the two words on all $\gamma \in [k+1]^n$ (see \cite[Proposition 1]{C}). 
\item If $u,v \in [2]^*$, then  $u\equiv_{\rm gram}v$ if and only if $u\equiv_{\rm plac}v$ (see \cite[Corollary 1]{C}).
\item The congruence $\equiv_{\rm gram}$ on $[3]^*$ is the least congruence generated by $\mathcal{R}_{{\rm plac}, 3} \cup \{(3212,2132)\}$ (see \cite[Theorem 1]{C}).
\end{itemize}

\section{A tropical representation of the grammic monoid}
Let $\mathbb{T}$ denote the tropical semiring consisting of the set $\mathbb{R} \cup \{-\infty\}$
under the operations $a \oplus b = {\rm max}(a, b)$ and $a \otimes b = a + b$ for all $a,b \in \mathbb{T}$, where $-\infty$ is the ``zero'' element, that is  $-\infty \oplus a = a \oplus -\infty = a$ and $-\infty \otimes a = a \otimes -\infty = -\infty$, for all $a \in \mathbb{T}$. We write ${\rm M}_n(\mathbb{T})$ to denote the monoid of all $n \times n$  matrices with entries from $\mathbb{T}$ under the matrix multiplication induced from operations of $\mathbb{T}$ in the obvious way and with identity element the $n \times n$ matrix with all diagonal equal to $0$ and
all other entries equal to $-\infty$, which we shall denote by $I_n$. We say
that $A \in M_n(\mathbb{T})$ is upper triangular if $A_{i,j} = -\infty$ for all $i > j$, and write
${\rm UT}_n(\mathbb{T})$ for the submonoid of $n \times n$ upper triangular matrices over $\mathbb{T}$. A tropical representation of a semigroup $S$ is a semigroup morphism $\varphi: S \rightarrow {\rm UT}_n(\mathbb{T})$.  We say that the representation is faithful if the morphism is injective.

The following definition will be crucial in all that follows.

\begin{definition}
For each word $w \in [n]^+$, and each pair $p, q \in [n]$ with $p \leq q$, let $w_{p,q}$ denote the maximal length of a weakly increasing subsequence
of $w$ with entries in the interval $[p,q]$.
\end{definition}

\begin{example}
If $w = 1535372549$ then $w_{3,5} = 3$; note that there are several maximal length weakly increasing subsequences of $w$ having entries in in the interval $[3,5]$, for example $555$, $334$, $335$, and $355$.
\end{example}

\begin{lemma}\cite[Lemma 5.1 and Lemma 5.2]{CKKMO}
\label{lem:rep}
The function $\varphi_n: [n]^* \rightarrow {\rm UT}_n(\mathbb{T})$ defined by $\varphi_n(\varepsilon) = I_n$ and for all $w \in [n]^+$ with $p, q \in [n]$,
$$\varphi_n(w)_{p,q} = \begin{cases}
   w_{p,q} & \mbox{ if }1 \leq p \leq  q \leq n\\
   -\infty & \mbox{ else}.
\end{cases}$$ is a monoid morphism. Moreover, if $u \equiv_{\rm plac} v$, then $\varphi_n(u) = \varphi_n(v)$.
\end{lemma}

\begin{remark}
\label{rem:rep}
The interest in tropical representations of the plactic monoid stems from work of Kubat and Okni\'{n}ski \cite[Theorem 2.6]{KO}, who proved that the plactic monoid of rank $3$ satisfies a non-trivial semigroup identity and conjectured this to be the case for all finite rank plactic monoids. The morphism defined above appeared in work of Cain, Klein, Kubat, Okni\'{n}ski  and the second author \cite{CKKMO}, who used this (together with another map) to construct a faithful tropical representation of the plactic monoid of rank $3$; a similar approach is taken by Izhakian \cite{I2} to exhibit a faithful tropical representation of the plactic monoid of rank $3$. Since the monoid of $n \times n$ upper triangular tropical matrices  satisfies non-trivial semigroup identities \cite{I} (see also \cite{O} for an alternative short proof), these faithful tropical representations of the plactic monoid of rank $3$ give an alternative proof of \cite[Theorem 2.6]{KO}. The morphism $\varphi_n$ is not injective for larger values of $n$ and so the methods of \cite{I2} and \cite{CKKMO} break down in higher ranks. However, by building on the underlying idea of counting maximal subsequences of certain kinds, a faithful tropical representation of the plactic monoid was later constructed by Kambites and the first author \cite{JK}, hence providing a proof of the conjecture of Kubat and Okni\'{n}ski \cite{KO}.
\end{remark}

\begin{definition}
\label{def:BandT}
Let $u \in [n]^+$. For $i=1, \ldots, n$ let $\beta(u,i)$ denote the number of $i$'s in the bottom row of $P(u)$, and (as above) let $u_{1,i}$ denote the maximal length of a weakly increasing subsequence of $u$ with entries in $\{1,\ldots, i\}.$ We define 
\begin{eqnarray*}
B(u) &=& (\beta(u,1), \ldots, \beta(u,n)) \in \mathbb{N}_0^n,\\
T(u) &=& (u_{1,1}, \ldots, u_{1,n}) \in \mathbb{N}_0^n.
\end{eqnarray*} Thus $B(u)$ is an encoding of the bottom row of tableau $P(u)$, and $T(u)$ is the top row of matrix $\varphi_n(u)$.
\end{definition}

\begin{remark}
\label{rem:action}
It follows from the definition above that for all $\gamma \in \mathbb{N}_0^n$ and all $u \in [n]^*$,
$$\gamma\cdot u = B(1^{\gamma_1}2^{\gamma_2} \cdots n^{\gamma_n} u).$$
\end{remark}
\begin{remark}
\label{rem:contraction}
Fix $u \in [n]^+$ and $1 \leq p \leq n$, and let $\bar{u}$ denote the word obtained from $u$ by deleting all occurrences of the symbols $1, \ldots, p-1$. Then for each $q \geq p$ it is clear that $u_{p,q} = \bar{u}_{p,q}$.  Moreover, since the least symbol occurring in $\bar{u}$ is greater than or equal to $p$, it follows that $\bar{u}_{p,q}$ is equal to the number of symbols in the bottom row of $P(\bar{u})$ that lie in the range $[p,q]$.
\end{remark}

\begin{example}
\label{example:store}
Let $n=8$ and $u=78463572256$. 
For $p=6$, we have $\bar{u} = 78676$. Thus
\begin{center}
$P(u) =$\begin{ytableau}
	7 & 8\\
	4 & 6\\
	3 & 5 & 7 \\
	2 & 2 & 5 &6
\end{ytableau} $P(\bar{u})=$ \begin{ytableau}
	8\\
	7 & 7\\
	6 & 6\\
\end{ytableau}
\end{center}
and the maximum length of a weakly increasing subsequence of $u$ over alphabet $[6,8]$ is $2$. Next, consider the action of $u$ on the bottom row containing exactly three $6$'s, which we identify with its row word $666$. To make the calculation transparent, we factorise $u$ into its row words, splitting each into two factors: symbols strictly less than $p$ and symbols greater than or equal to $p$, writing $u= 78 \, 4\, 6\, 35\, 7\, 225\, 6$. It is then easy to see that 
$$666 \rightarrow 666\, 78 \rightarrow 466\, 78 \rightarrow 466\, 68 \rightarrow 356\, 68 \rightarrow 356\, 67 \rightarrow 225\, 67 \rightarrow 225\, 66.$$
Notice that $225$ coincides with the symbols strictly less than $p=6$ from the bottom row of $P(u)$ and $66$
coincides with the bottom row of $P (\bar{u})$.
\end{example}

\begin{remark}
\label{rem:bottomrow}
Consider the notation of Remark \ref{rem:contraction}. The calculation in Example \ref{example:store} can be generalised to show that for any $u \in [n]^+$ and any $1 \leq p \leq n$, the action of $u$ on the bottom row consisting of $N:=\sum_{i=1}^{p-1}\beta(u,i)$ symbols $p$ may be computed by first inserting (using Schensted's row insertion algorithm) the symbols from $u$ that are strictly less than $p$ and then  acting by $\bar{u}$ on this bottom row. Note that each symbol less than $p$ will only bump symbols one of the $N$ symbols $p$. Thus this action can be reconstructed from knowledge of the bottom row of $P(\bar{u})$ together with the symbols in the bottom row of $P(u)$ that are strictly less than $p$.
\end{remark}

\begin{remark}
\label{rem:toprow}
Let $w, u \in [n]^*$. Since $\varphi_n$ is a morphism, $\varphi_n(wu) = \varphi_n(w) \otimes \varphi_n(u)$. The top row $T(wu)$ of $\varphi_n(wu)$ is then given by
$$T(wu) = (T(w) \otimes C_1(u), \ldots, T(w) \otimes C_n(u)),$$
where $C_i(u)$ denotes the $i$th column of $\varphi_n(u)$.
\end{remark}

For any word $u \in [n]^*$, the top row $T(u)$ of the matrix $\varphi_n(u)$ contains no $-\infty$ entries, and may therefore be viewed as a rational vector. For a rational $n \times n$ matrix $X$ one can then form the (usual) matrix product $T(u)X$. In the following we write $e_i$ to denote the $i$th standard basis vector of $\mathbb{Q}^n$. 
\begin{lemma}
\label{lem:botrow}
Let $u \in [n]^+$. Then $B(u) = T(u)X$, where $X$ is the matrix with first column $e_1$ and with $i$th column equal to $e_i-e_{i-1}$, for $i\geq 2$.
\end{lemma}

\begin{proof}
Since $\varphi_n$ is constant on $\equiv_{\rm plac}$-classes, we may assume without loss of generality that $u$ is the row reading of $P(u)$. By definition, $u_{1,i}$ is the maximal length of a weakly increasing subsequence of $u$ with entries in $\{1, \ldots, i\}$. Since $u$ is the row word of $ P(u)$, it is easy to see that $u_{1,i}$ is equal to the number of entries less than or equal to $i$ in the bottom row of $P(u)$ -- see \cite[Lemma 2.8]{JK} for a short proof. The number of $1$'s in the bottom row is therefore $u_{1,1}$, whilst for all $ i \geq 2$ we have $u_{1,i} - u_{1,i-1}$ is the number of $i$'s in the bottom row of $P(u)$. The result now follows.
\end{proof}

\begin{thm}
\label{thm:sequences}
Let $u,v \in [n]^*$. Then $u \equiv_{\rm gram} v$ if and only if $\varphi_n(u) = \varphi_n(v)$. (In other words, if and only if for all sub-intervals $[p,q]$ of $[n]$, the maximum length of a weakly increasing subsequence of $u$ over alphabet $[p,q]$ is equal to  the maximum length of a weakly increasing subsequence of $v$ over alphabet $[p,q]$.)
\end{thm}
\begin{proof}
Suppose that $\varphi_n(u) = \varphi_n(v)$. It follows from the description of $\varphi_n$ that if either of $u$ or $v$ is equal to the empty word, then $u=v=\varepsilon$. Thus we may assume that $u, v \in [n]^+$.  Let $\gamma \in \mathbb{N}_0^n$ and let $w = 1^{\gamma_1}2^{\gamma_2} \cdots n^{\gamma_n}$. Then by Remark \ref{rem:action}, Lemma \ref{lem:botrow}, Remark \ref{rem:toprow} together with the fact that $C_i(u) = C_i(v)$, for all $i$, we have: 
\begin{eqnarray*}
\gamma \cdot u &=& B(wu) = T(wu)X = (T(w) \otimes C_1(u), \ldots, T(w) \otimes C_n(u))X\\
&=& (T(w) \otimes C_1(v), \ldots, T(w) \otimes C_n(v))X =  T(wv)X = B(wv) = \gamma \cdot v.\end{eqnarray*}
Since $\gamma$ was arbitrary, this shows that $u \equiv_{\rm gram} v$.

Suppose that $u \equiv_{\rm gram} v$.  If either of $u$ or $v$ is equal to the empty word, then it is immediate that $u=v=\varepsilon$. Thus we may assume that $u,v \in [n]^+$. Fix $p \in \{1, \ldots, n\}$ and let $\bar{u}$ and $\bar{v}$ denote the words obtained from $u$ and, respectively, $v$ by deleting all occurrences of the symbols $1, \ldots, p-1$. Now, since $u \equiv_{\rm gram} v$ we know that $P(u)$ and $P(v)$ have the same bottom row and hence in particular, the number of symbols from $1, \ldots, p-1$ in the bottom row of $P(u)$ is equal to the number of symbols from $1, \ldots, p-1$ in the bottom row of $P(v)$; call this number $n_p$.  Let $\delta \in \mathbb{N}_0^n$ be the row of length $n_p$ containing only the symbol $p$ and consider the action of $u$ and $v$ on $\delta$. It follows from the definition of the action that any symbol strictly less than $p$ will bump one of the first $n_p$ symbols in this row, whilst each symbol of $\bar{u}$ (which is by definition greater than or equal to $p$) will be inserted to the right of the first $n_p$ symbols as though inserting $\bar{u}$ into the empty tableau (see Remark \ref{rem:bottomrow}). Thus $\delta \cdot u$ is equal to the row formed as the concatenation of the first $n_p$ symbols of $P(u)$ followed by the bottom row of $P(\bar{u})$. Likewise, $\delta \cdot v$ is equal to the row formed as the concatenation of the first $n_p$ symbols of $P(v)$ followed by the bottom row of $P(\bar{v})$. Since $u \equiv_{\rm gram} v$, we deduce from $\delta \cdot u = \delta\cdot v$ that the bottom rows of $P(\bar{u})$ and $P(\bar{v})$ are equal. Thus by Remark \ref{rem:contraction}, for all $q \geq p$ we have 
$$u_{p,q} = \bar{u}_{p,q} =  \bar{v}_{p,q} = v_{p,q}.$$
Since this statement holds for all $1 \leq p \leq q$, we have $\varphi_n(u) = \varphi_n(v)$.
\end{proof}

\begin{remark}
Theorem \ref{thm:sequences} states that $\varphi_n$ restricts to give a faithful representation of the grammic monoid of rank $n$ by $n \times n$ upper triangular tropical matrices.
\end{remark}

\begin{remark}
The description of the relation $\equiv_{\rm gram}$ given by Theorem \ref{thm:sequences} makes clear that the grammic congruence coincides with the 
`cloaktic congruence' of Izhakian \cite{I2} where a linear (in the length of the input word $u \in [n]^*$, treating the size of the alphabet as constant) time    algorithm \cite[Algorithm 5.13]{I2} is provided to compute $\varphi_n(u)$.

It also follows from \cite[Proposition 1]{C} that the word problem for the grammic monoid of any finite rank is decidable; Theorem \ref{thm:sequences} and the discussion above gives an alternative proof of this fact.
\end{remark}

\section{Semigroup varieties generated by a single finite rank grammic monoid}

A semigroup identity on alphabet $\Sigma$ is a pair of words, written $u = v$, in the
free semigroup $\Sigma^+$. The identity $u=v$ is said to be non-trivial if $u$ and $v$ are distinct words, and balanced if $|u|_a=|v|_a$ for all $a \in \Sigma$. For a balanced identity $u=v$, we in particular have that $|u|=|v|$; we call this common value the length of the identity. We say that $u=v$ holds in a semigroup $S$
(or $S$ satisfies the identity) if every morphism from $\Sigma^+$ to $S$ maps $u$ and $v$ to the
same element of $S$. The variety  generated by a semigroup $S$ is
the class of semigroups that satisfy all the semigroup identities satisfied by $S$.

We recall from the introduction that for each $n \geq 1$ we write $\mathbf{V}_n$ to denote the variety of semigroups generated by the monoid of $n \times n$ upper triangular tropical matrices. Combining Theorem \ref{thm:sequences} with the results of \cite{JK} now yields:

\begin{corollary}\label{cor:identities}
The grammic monoid of rank $n$ satisfies exactly the same identities as the monoid of $n \times n$ upper triangular tropical matrices. In other words: the variety of semigroups generated by the grammic monoid of rank $n$ is $\mathbf{V}_n$.
\end{corollary}

\begin{proof}
Since the grammic monoid of rank $n$ has a faithful representation by $n \times n$ upper triangular tropical matrices, it is immediate that every identity satisfied by ${\rm UT}_n(\mathbb{T})$ is also satisfied by the grammic monoid of rank $n$.

For the converse, we note that the proof of \cite[Theorem 4.4]{JK} demonstrates that if an identity can be falsified by matrices in ${\rm UT}_n(\mathbb{T})$, then it can be falsified by matrices lying in the image of $\varphi_n$. Since $\varphi_n$ is a faithful representation of the grammic monoid, we have that the identity also does not hold in the grammic monoid.
\end{proof}

\begin{remark}
\label{rem:Volkov}
We note that \cite[Theorem 1]{V} states that the grammic monoid of rank $3$ satisfies the same semigroup identities as the plactic monoid of rank 3 (and also as the quotient of the plactic monoid of rank $3$ factored by the least congruence identifying $13$ and $31$). Thus Corollary  \ref{cor:identities} may be viewed as a generalisation of \cite[Theorem 1]{V}. 
\end{remark}

As mentioned in the introduction, the varieties $\mathbf{V}_n$ have received much attention. Corollary \ref{cor:identities} together with the results of \cite{AR}, \cite{A}, \cite{AR3}, \cite{DJK} and \cite{JF} yields the following immediate consequences:

\begin{corollary}
The variety generated by the stylic monoid of rank $n$ is a proper subvariety of the variety generated by the grammic monoid of rank $n+1$. Consequently, for all $n \geq 4$ the variety of the grammic monoid of rank $n$ has uncountably many subvarieties. 
\end{corollary}

\begin{proof}
By \cite[Corollary 4.2]{AR3} and \cite[Corollary 3.4]{JF}, the stylic monoid and the monoid of uni-triangular $(n+1) \times (n+1)$ tropical matrices each generate the same variety $\mathbf{J}_n$. Since the uni-triangular matrices form a subsemigroup of the upper triangular matrices of the same size, it follows from Corollary~\ref{cor:identities} that this is a subvariety of the variety generated by the grammic monoid of rank $n+1$. By \cite[Corollary 4.4]{AR3}, the variety generated by the stylic monoid of rank $n$ has uncountably many subvarieties for $n \geq 3$.
\end{proof}

\begin{corollary}
\label{cor:identitychecking}
Let $n$ be a fixed positive integer. There is an algorithm which,
given an identity $v = w$ over a finite alphabet $\Sigma$, decides in time polynomial in $|v| + |w|$
and $|\Sigma|$ whether the identity holds in the grammic monoid of rank $n$.
\end{corollary}

\begin{proof}
This follows immediately from Corollary \ref{cor:identities} and \cite[Theorem 8.3]{DJK}.
\end{proof}

\begin{corollary}
\label{cor:varieties}
Let $m, n$ be positive integers. The variety of semigroups generated by the grammic monoid of rank $n$ coincides with the variety of semigroups generated by the grammic monoid of rank $m$ if and only if $m=n$.
\end{corollary}

\begin{proof}
Suppose that $n<m$. By \cite[Theorem 3.4]{A},  there exists an identity satisfied by ${\rm UT}_n(\mathbb{T})$ (and hence by the grammic monoid of rank $n$) that is not satisfied by ${\rm UT}_m(\mathbb{T})$ (and hence not satisfied by  the grammic monoid of rank $m$). 
\end{proof}

\begin{corollary}\label{prop:increasingidentities}
The identities of each finite rank grammic monoid are balanced, and the grammic monoid of rank $n$ does not satisfy any non-trivial semigroup identity of length less than or equal to $n-1$.
\end{corollary}

\begin{proof}
Suppose for contradiction that the grammic monoid of rank $n$ satisfies a non-trivial semigroup  identity $u=v$ of length $\ell \leq n-1$. Thus  $\ell =|u|=|v|$. By Corollary \ref{cor:identities}, the identity $u=v$ is satisfied by the monoid of $n \times n$ upper triangular tropical matrices, and hence also by the subsemigroup of $n \times n$ \emph{uni-triangular} tropical matrices. It then follows from \cite[Corollary 3.3]{JF} that $u$ and $v$ admit the same set
of scattered subwords of length at most $n-1$, and since $\ell < n$, this immediately forces $u=v$, giving a contradiction. 
\end{proof}

\begin{corollary}
\label{cor:infrank}
The infinite rank grammic monoid does not satisfy any non-trivial semigroup identity.
\end{corollary}

\begin{proof}
For each positive integer $n$, we show that the grammic monoid of rank $n$ is a submonoid of the infinite rank grammic monoid; the result then follows from Corollary \ref{prop:increasingidentities}. It suffices to show that $u,v \in [n]^*$ act the same on all bottom rows over alphabet $\mathbb{N}$ if and only if they act the same on all bottom rows over alphabet $[n]$. The forward direction is clear.

Suppose then that $\gamma \cdot u = \gamma \cdot v$ holds for all bottom rows $\gamma$ on alphabet $[n]$, and let $\delta \in \mathbb{N}_0^{n+1}$ be a bottom row over alphabet $[n+1]$. We show that $\delta \cdot u = \delta \cdot v$.  Let $\delta' \in \mathbb{N}_0^{n}$ be the restriction of $\delta$ to the first $n$ symbols. Setting $d= \sum_{i=1}^n (\delta'_i-(\delta' \cdot u)_i) = \sum_{i=1}^n (\delta'_i-(\delta' \cdot v)_i)$ note that the action of either $u$ or $v$ on $\delta'$ results in a row with $d$ more symbols. For $1 \leq i \leq n$ it is clear that $(\delta \cdot u)_i = (\delta' \cdot u)_i = (\delta' \cdot v)_i = (\delta \cdot v)_i$. Moreover, since $u, v \in [n]^*$, the action of these words can only decrease the number of entries $n+1$ in $\delta$ by at most $d$ and so
$$(\delta \cdot u)_{n+1} = (\delta \cdot v)_{n+1} = \begin{cases}
\delta_{n+1} - d & \mbox{ if } \delta_{n+1} \geq d\\
0 & \mbox{ else.}
    \end{cases}
$$
Thus $u$ and $v$ act the same on all bottom  rows over $[n+1]$. Since $u,v \in [n]^* \subseteq [n+1]^*$, we may then apply induction to obtain that $u$ and $v$ act the same on all bottom rows.
\end{proof}

\section{Compatibility properties}

Several important congruences on words are known to satisfy various compatibility properties; for example, the plactic congruence is known to be compatible with restriction to alphabet intervals, packing, standardisation, and Sch\"utzenberger involution (see \cite{LS} for details). Other monoids closely related to the plactic monoid are also known to satisfy these properties - see \cite{AR4} for further results.
In this section we show that the grammic congruence also satisfies these properties.

Let $\mathcal{A}$ be an ordered alphabet and $\mathcal{B}$ a subalphabet of $\mathcal{A}$. We say that $\mathcal{B}$ is a subinterval if for all $b_1, b_2 \in \mathcal{B}$ if $a \in \mathcal{A}$ with $b_1< a <b_2$, then $a \in \mathcal{B}$. For $u \in \mathcal{A}^*$ with support $\{a_1<\ldots <a_m\}$ and $k = |u|$ we write:
\begin{itemize}
\item ${\rm cont}(u)$ to denote the content of $u$, that is, the function mapping each $a \in \mathcal{A}$ to $|u|_a$;
\item $u_\mathcal{B}$ to denote the restriction of $u$ to the subalphabet $\mathcal{B}$, that is, the word obtained from $u$ be deleting each symbol not in $\mathcal{B}$;
\item ${\rm pack}(u)$ to denote the packed word of $u$ that is, the word in $[m]^*$ determined by replacing each $a_i$ by $i$; and
\item ${\rm std}(u)$ to denote the standardisation of $u$, that is, the permutation of $[k]$ determined by replacing for each $1 \leq i \leq m$ the $j$th instance of $a_i$ in $u$ (reading from left to right) by $j+\Sigma_{t=1}^{i-1}|u|_{a_{t}}$.
\end{itemize}
If $\mathcal{A} = \{1, \ldots, n\}$ for some natural number $n$, then we also write $u^*$ to denote the Sch\"{u}tzenberger involution of $u$, that is the word obtained from $u$ by replacing each letter $i$ by $n-i+1$ and then reversing the word. For a congruence $\equiv$  on $\mathcal{A}^*$, the factor monoid $\mathcal{A}^*/\equiv$ is then said to be:
\begin{itemize}
\item compatible with restriction to alphabet intervals if: $$\forall u,v \in \mathcal{A}^*, u\equiv v \Leftrightarrow u_\mathcal{I} \equiv v_\mathcal{I} \mbox{  for all subintervals }\mathcal{I};$$
\item compatible with packing if: 
$$\forall u,v \in \mathcal{A}^*, u\equiv v \Leftrightarrow  ({\rm pack}(u) \equiv {\rm pack}(v) \mbox{ and }{\rm cont}(u) = {\rm cont}(v));$$
\item compatible with standardisation if:$$ \forall u, v \in \mathcal{A}^*, u\equiv v \Leftrightarrow  ({\rm  std}(u) \equiv {\rm std}(v) \mbox{ and } {\rm cont}(u) = {\rm cont}(v));$$
\item compatible with the Sch\"{u}tzenberger involution if $\mathcal{A}$ is finite and: 
$$\forall u,v \in \mathcal{A}^*, u\equiv v \Leftrightarrow  u^*\equiv v^*.$$
\end{itemize}

\begin{prop}
The grammic monoid of rank $n$ is compatible with restriction to alphabet intervals, packing, standardisation,  and Sch\"utzenberger involution.
\end{prop}

\begin{proof}

Throughout let $u, v\in [n]^*$.  

\medskip
\textbf{Restriction to subintervals:} If $u_\mathcal{I} \equiv_{\rm gram} v_\mathcal{I}$ holds for all subintervals of $\mathcal{A}$, then in particular it holds for the interval $\mathcal{I} =\mathcal{A}$, giving $u \equiv_{\rm gram} v$. Now let  $\mathcal{I} = [s,t]$ be a subinterval of $[n]^*$ suppose that $u \equiv_{\rm gram} v$, aiming to show that $u_\mathcal{I} \equiv_{\rm gram} v_\mathcal{I}$. By Theorem \ref{thm:sequences} it suffices to show that $\varphi_n(u_\mathcal{I})_{p,q}=\varphi_n(u_\mathcal{I})_{p,q}$ for all $1 \leq p \leq q \leq n$. If $q<s$ or $p>t$, then this is immediate (there are no weakly increasing subsequences over alphabet $[p,q]$). Suppose then that $[p,q] \cap [s,t]$ is non-empty. In this case, it is clear that the weakly increasing subsequences of $u_\mathcal{I}$ (respectively, $v_\mathcal{I}$) over alphabet $[p,q]$ coincide with the weakly increasing subsequences of $u_\mathcal{I}$ (respectively, $v_{\mathcal{I}}$) over alphabet $[{\rm max}(p,s),{\rm min}(q,t)]$ which in turn coincide with the weakly increasing subsequences of $u$ (respectively, $v$) over alphabet $[{\rm max}(p,s),{\rm min}(q,t)]$. Thus
by Theorem \ref{thm:sequences} we have:
\begin{eqnarray*}
\varphi_n(u_\mathcal{I})_{p,q} = \varphi_n(u_\mathcal{I})_{{\rm max}(p,s),{\rm min}(q,t)} &=& \varphi_n(u)_{{\rm max}(p,s),{\rm min}(q,t)}\\ &=& \varphi_n(v)_{{\rm max}(p,s),{\rm min}(q,t)} \\
 & = & \varphi_n(v_\mathcal{I})_{{\rm max}(p,s),{\rm min}(q,t)} = \varphi_n(v_\mathcal{I})_{p,q},
\end{eqnarray*}
giving $u_\mathcal{I} \equiv_{\rm gram} v_\mathcal{I}$.

\medskip
\textbf{Packing:} Suppose $u = u_1\cdots u_k \in [n]^*$ where $u_i \in [n]$ and let $\mathcal{U}:=\{a_1<\cdots <a_m\}$ denote the support of $u$. For each $a_i \in \mathcal{U}$ let $\hat{a_i} = i$, so that ${\rm pack}(u) = \hat{u_1}\cdots \hat{u_k} \in [m]^*$. Then for all $1 \leq p \leq q \leq m$, the weakly increasing subsequences of ${\rm pack}(u)$ over alphabet $[p,q]$ are in length-preserving bijection with the weakly increasing subsequences of $u$ over alphabet $[a_p,a_q]$, so that
\begin{eqnarray}
\label{eq:hat1}
\varphi_m({\rm pack}(u))_{p,q} &=& \varphi_n(u)_{a_p,a_q}. 
\end{eqnarray}
Similarly, for all $1 \leq p \leq q \leq n$ such that $[p,q] \cap \mathcal{U} \neq \emptyset$, we have  
\begin{eqnarray}
\label{eq:hat2}
\varphi_n(u)_{p,q} &=& \varphi_m({\rm pack}(u))_{p',q'} 
\end{eqnarray}
where $p' = {\rm min}\{j: p \leq a_j\}$ and $q' = {\rm max}\{j: a_j \leq q\}$ again depend only on the support of $u$ and the chosen $p$ and $q$. Whilst if $1 \leq p \leq q \leq n$ is such that $[p,q] \cap \mathcal{U} =\emptyset$, we have $\varphi_n(u)_{p,q} = 0$.

Now, if $u\equiv_{\rm gram} v$ then by \cite[Lemma 1]{C} we have ${\rm cont}(u)={\rm cont}(v)$, and hence in particular $v$ also has support $\mathcal{U}$ and length $k$. By equation \eqref{eq:hat1} applied to both $u$ and $v$ and Theorem \ref{thm:sequences}, for all $1 \leq p \leq q \leq m$
$$\varphi_m({\rm pack}(u))_{p,q} = \varphi_n(u)_{a_p,a_q} = \varphi_n(v)_{a_p,a_q} = \varphi_m({\rm pack}(v))_{p,q},$$
giving ${\rm pack}(u) \equiv_{\rm gram} {\rm pack}(v)$. Conversely, suppose that ${\rm pack}(u) \equiv_{\rm gram} {\rm pack}(v)$ and ${\rm cont}(u)={\rm cont}(v)$. Then $u$ and $v$ have the same length and the same support set $\mathcal{U}$. Let $1 \leq p \leq q \leq n$. 
If $[p,q] \cap \mathcal{U} =\emptyset$, then we immediately have $\varphi_n(u)_{p,q} = 0 = \varphi_n(v)_{p,q}$. Otherwise, $[p,q] \cap \mathcal{U}  \neq \emptyset$ and by equation  \eqref{eq:hat2} applied to both $u$ and $v$ and Theorem \ref{thm:sequences},
$$\varphi_n(u)_{p,q} = \varphi_m({\rm pack}(u))_{p',q'} = \varphi_m({\rm pack}(v))_{p',q'} = \varphi_n(v)_{p,q},$$
giving $u \equiv_{\rm gram} v$.

\medskip
\textbf{Standardisation:} Let $u = u_1 \cdots u_k$, where each $u_i \in [n]$, and let ${\rm std}(u) = u_1'\cdots u_k' \in [k]^*$. By construction of ${\rm std}(u)$ for all $1 \leq p \leq q \leq n$, the sequence $u_{i_1} \leq u_{i_2} \leq \cdots \leq u_{i_t}$ is a weakly increasing subsequence of $u$ over the interval $[p,q]$ if and only if $u_{i_1}' < u_{i_2}' < \cdots < u_{i_t}'$ is a (weakly) increasing subsequence of ${\rm std}(u)$ over the interval $[a,b]$ where $a = 1+\sum_{i<p} |u|_i$ and $b=k-\sum_{i>q} |u|_i$. Note that $a$ and $b$ are determined by $p, q$ and the content of $u$ only and, moreover, $a \leq b$ holds provided that there is at least one non-empty subsequence of $u$ over the interval alphabet $[p,q]$ --  in this case we have $\varphi_n(u)_{p,q} = \varphi_k({\rm std}(u))_{a,b}$. On the other hand, if $u$ does not contain any weakly increasing subsequence over this interval, then $u$ in particular contains no letters from this interval.

Now suppose that ${\rm std}(u) \equiv_{\rm gram} {\rm std}(v)$ and ${\rm cont}(u)={\rm cont}(v)$. Then for all $1 \leq p \leq q \leq n$ either $u$ (and hence also $v$) contains a letter from $[p,q]$ giving 
$$\varphi_n(u)_{p,q} = \varphi_k({\rm std}(u))_{a,b} = \varphi_k({\rm std}(v))_{a,b} = \varphi_n(v)_{p,q},$$
or else $u$ (and hence also $v$) contains no letter from $[p,q]$ giving $\varphi_n(u)_{p,q} = 0 = \varphi_n(v)_{p,q}$. Then by Theorem \ref{thm:sequences} $u \equiv_{\rm gram} v$. 

Conversely, suppose that $u \equiv_{\rm gram} v$. By \cite[Lemma 1]{C} we have ${\rm cont}(u)={\rm cont}(v)$.
As before, let us write $u=u_1 \cdots u_k$, ${\rm std}(u) = u_1'\cdots u_k'$, $v=v_1 \cdots v_k$ and ${\rm std}(v) = v_1'\cdots v_k'$, and let  $1 \leq a \leq b \leq k$ be arbitrary. We show that $\varphi_k({\rm std}(u))_{a,b} \leq \varphi_k({\rm std}(v))_{a,b}$.

Suppose then that $u_{i_1}' < \cdots < u_{i_t}'$ is a maximum length (weakly) increasing subsequence of ${\rm std}(u)$ over the interval $[a,b]$. Setting $a' = u_{i_1}'$ and $b' = u_{i_t}'$ we have $a \leq a' \leq b' \leq b$ and 
\begin{equation}
\label{eq:1}
\varphi_k({\rm std}(u))_{a,b} = \varphi_{k}({\rm std}(u))_{a',b'}.
\end{equation}
Without loss of generality suppose that (reading from left to right) $u_{i_1}$ is the $r$th $p$ of $u$ and $u_{i_t}$ is the $s$th $q$ of $u$ for some $p,q \in [n]$, and $1 \leq r \leq |u|_p=|v|_p$, $1 \leq s \leq |u|_q=|v|_q$. Then $a'= r + \sum_{x=1}^{p-1} |u|_x$, $b'= s + \sum_{x=1}^{q-1} |u|_x$ and $u_{i_1} \leq \cdots \leq u_{i_t}$ is a weakly increasing subsequence of $u$ over the interval $[p,q]$ starting on the $r$th $p$ of $u$ and ending on the $s$th $q$ of $u$. By including the first $r-1$ symbols $p$ of $u$ (each of which appears to the left of $u_{i_1})$ and the last $(|u|_q-s)$ symbols $q$ of $u$ (each of which appears to the right of $u_{i_t}$), we see that
$$\underbrace{p \leq \cdots \leq p}_{r-1} \leq u_{i_1} \leq \cdots \leq u_{i_t} \leq \underbrace{q \leq \cdots \leq q}_{|u|_q-s}$$
is also a weakly increasing subsequence of $u$ over the interval $[p,q]$. Thus
$$\varphi_n(u)_{p,q} \geq  \varphi_k({\rm std}(u))_{a'b' 
} +(r-1) + (|u|_q-s)$$
which together with the fact that $u \equiv_{\rm gram} v$ gives \begin{equation}
\label{eq:2}
\varphi_{k}({\rm std}(u))_{a',b'} \leq  \varphi_n(u)_{p,q} - (r-1) - (|u|_q-s)=\varphi_n(v)_{p,q} - (r-1) - (|v|_q-s).
\end{equation}

Now let $v_{j_1} \leq \cdots \leq v_{j_{\ell}}$ be a maximum length weakly increasing subsequence of $v$ over the interval $[p,q]$. As in the first paragraph, setting $e=1+\sum_{i<p} |v|_i$ and $f=k-\sum_{i>q} |v|_i$ we see that $v_{j_1}' < \cdots < v_{j_{\ell}}'$ is an increasing sequence over alphabet $[e,f]$. Moreover, since $u$ and $v$ have the same content, we see that $e \leq a' \leq b' \leq f$. Let $L \geq 0$ be the length of the uniquely determined (but possibly empty) maximal increasing subsequence of $v_{j_1}' < \cdots < v_{j_{\ell}}'$ over alphabet $[a',b']$. Noting that this subsequence is obtained by removing all symbols $v_{j_i}'$ with the property that $v_{j_i}$ is equal to either one of the first $r-1$ symbols $p$ of $v$ or one of the last $|v|_q-s$ symbols $p$ of $v$, we have
$$L \geq \ell - (r-1) - (|v|_q-s) = \varphi_n(v)_{p,q} - (r-1) - (|u|_q-s),$$ and hence
\begin{equation} 
\label{eq:3}\varphi_k({\rm std}(v))_{a'b'} \geq L \geq \varphi_n(v)_{p,q} - (r-1) - (|v|_q-s).\end{equation}
By equations \eqref{eq:1}, \eqref{eq:2} and \eqref{eq:3} together with the fact that $a \leq a' \leq b' \leq b$
we then have
\begin{eqnarray*}
\varphi_k({\rm std}(u))_{a,b} &=& \varphi_k({\rm std}(u))_{a',b'}\\ &\leq& \varphi_n(v)_{p,q} - (r-1)-(|v|_q-s) \\
&\leq& \varphi_k({\rm std}(v))_{a'b'}\\
&\leq& \varphi_k({\rm std}(v))_{a,b}.
\end{eqnarray*}
A dual argument interchanging the roles of $u$ and $v$ shows that $\varphi_k({\rm std}(v))_{a,b} \leq \varphi_k({\rm std} (u))_{a,b}$. Since $a,b$ were arbitrary, this gives ${\rm std}(u) \equiv_{\rm gram} {\rm std}(v)$.

\medskip
\textbf{Sch\"utzenberger involution:} Suppose that $u=u_1\cdots u_k$ where each $u_i \in [n]$, and notice that for all $1 \leq p \leq q \leq n$
\begin{eqnarray*}
    \varphi_n(u)_{p,q} &=& \mbox{ max. length of a weakly increasing subsequence of } u_1\cdots u_k \mbox{ over } [p,q]\\
    &=& \mbox{ max. length of a weakly decreasing subsequence of } u_k\cdots u_1\mbox{ over } [p,q]\\
    &=& \mbox{ max. length of a weakly increasing subsequence of } (n-u_k+1)\cdots (n-u_1+1)\\&& \mbox{ over } [n-q+1,n-p+1]\\
    &=& \varphi_n(u^*)_{n-q+1,n-p+1}.
\end{eqnarray*}
That is, $\varphi_n(u^*)$ is equal to the matrix obtained from $\varphi_n(u)$ by flipping along the anti-diagonal; call this matrix $(\varphi_n(u))^\Delta$.  
Then:  
$$u \equiv_{\rm gram} v \Leftrightarrow \varphi_n(u) = \varphi_n(v) \Leftrightarrow (\varphi_n(u))^\Delta = (\varphi_n(v))^\Delta \Leftrightarrow \varphi_n(u^*) = \varphi_n(v^*) \Leftrightarrow u^* \equiv_{\rm gram} v^*.$$

\end{proof}
\subsection{Relationship to charge sequences}
For a standard word $w$ on alphabet $[k]$ we define the charge sequence of $w$ denoted ${\rm chseq}(w)$ to be the word of length $k=|w|$ over the alphabet $\{0, \ldots, k-1\}$ defined iteratively by ${\rm chseq}(w)_1=0$ and for all $i \geq 2$ 
$${\rm chseq}(w)_i = \begin{cases} {\rm chseq}(w)_i & \mbox{ if } i \mbox{ occurs to the left of } i-1 \mbox{ in }w \\
{\rm chseq}(w)_i +1& \mbox{ if }  i \mbox{ occurs to the right of } i-1 \mbox{ in }w.
\end{cases}$$
The charge of $w$ is the sum of the individual entries in the charge sequence.

\begin{example}
If $w=82456137$, then ${\rm chseq}(w) = 00112344$ and the charge of $w$ is $15$. 
\end{example}

\begin{prop}
Let $u$ and $v$ be  standard words. If
$u \equiv _{\rm gram} v$ then ${\rm chseq}(u) = {\rm chseq}(v)$.
\end{prop}
\begin{proof}
Let $u$ and $v$ be standard words of length $k$, over alphabet $[k]$. Since $u$ and $v$ are standard, each letter $i \in [k]$ appears exactly once giving $\varphi_k(u)_{i,i} = \varphi_k(v)_{i,i} = 1$ for all $i=1, \ldots, k$.

Suppose that $u \equiv_{\rm gram} v$. By Theorem \ref{thm:sequences} we know that $\varphi_k(u) = \varphi_k(v)$. In particular, we have $\varphi_k(u)_{i-1,i} = \varphi_k(v)_{i-1,i}$ for all $i=2, \ldots, k$. Moreover, this common value is $1$ if $i$ occurs to the left of $i-1$ in both $u$ and $v$, and $2$ if $i$ occurs to the right of $i-1$ in both $u$ and $v$. Thus, for all $i \geq 2$ we have \begin{eqnarray*}
\varphi_k(u)_{i-1,i} - \varphi_k(u)_{i-1,i-1} &=& \begin{cases} 0 & \mbox{ if } i \mbox{ occurs to the left of } i-1 \mbox{ in }u \\
1& \mbox{ if }  i \mbox{ occurs to the right of } i-1 \mbox{ in }u
\end{cases}\\
&=&\begin{cases} 0 & \mbox{ if } i \mbox{ occurs to the left of } i-1 \mbox{ in }v \\
1& \mbox{ if }  i \mbox{ occurs to the right of } i-1 \mbox{ in }v
\end{cases}\\ &=&\varphi_k(v)_{i-1,i} - \varphi_k(v)_{i-1,i-1}.
\end{eqnarray*}
Then ${\rm chseq}(u)_1=0={\rm chseq}(v)_1$ and for all $j \geq 2$  it follows that
$${\rm chseq}(u)_j = \sum_{i=2}^j \varphi_k(u)_{i-1,i} - \varphi_k(u)_{i-1,i-1} = \sum_{i=2}^j \varphi_k(v)_{i-1,i} - \varphi_k(v)_{i-1,i-1} = {\rm chseq}(v)_j$$
giving ${\rm chseq}(v) = {\rm chseq}(v)$, as required.
\end{proof}
The converse does not hold however; for example taking $u = 3412$ and $v=1324$ one finds that ${\rm chseq}(u) = {\rm chseq}(v) = 0112$. However, the maximum length of a weakly increasing subsequence of $u$ is $2$, whilst the maximum length of a weakly increasing subsequence of $v$ is $3$, demonstrating that the two words are not grammic equivalent.

\begin{remark}
\label{rem:cyclic}
In \cite{CM}, the maximum diameter $d(M)$ of a connected component of the cyclic shift graph of various plactic-like monoids $M$ was studied, with exact values computed for the hypoplactic, sylvester, taiga and stalactic monoids. Upper and lower bounds were provided in the case of the plactic monoid $\mathbb{P}_n$, specifically \cite[Proposition 4.3]{CM} states that 
$$n-1\leq d(\mathbb{P}_n) \leq 2n-3.$$ The argument given to establish the lower bound of $n-1$ made use of the notion of co-charge sequence; a similar argument holds for the case of the grammic monoid using charge sequences. Moreover, since the grammic monoid $\mathbb{G}_n$ is a quotient of the plactic monoid $\mathbb{P}_n$, it is clear that $n-1 \leq d(\mathbb{G}_n) \leq d(\mathbb{P}_n) \leq 2n-3$ also.
\end{remark}

\section{Relations in the grammic monoid}
In \cite{C} a finite presentation for the grammic monoid of rank $3$ is given, but the general case is left open. In this section we consider some natural generalisations of 
the Knuth relations of the plactic monoid, the relations of the patience sorting monoid, and the two-column relations of the plactic monoid \cite{CGM15} that may provide the key to answering this question.

\subsection{Generalised Knuth relations and patience sorting monoids}
As observed in \cite{C}, if $u$ and $v$ are equivalent words in the plactic monoid of rank $n$, then $u \equiv_{\rm gram} v$. In particular, each of the Knuth relations of rank $n$ is satisfied in the grammic monoid. Recall that the set $\mathcal{R}_{{\rm plac},[n]}$ of Knuth relations of rank $n$ splits into two families:
$$\mathcal{R}_{{\rm plac}, [n]}:=\{(yz x,yxz):  x,y, z \in [n], x < y \leq z \} \cup \{(zxy,xzy): x,y, z \in [n], x \leq  y < z \}.$$
In this subsection we prove that the grammic monoid satisfies a set of relations generalising the first family of Knuth relations. 

\begin{prop} 
\label{prop:rel}
Let $u,v, w \in [n]^*$ and let $y$ be the least letter occurring in $w$. Suppose that $|w|_y \geq |v|$ and for each letter $x$ occurring in $v$ and each letter $z$ occurring in $u$ we have $x<y \leq z$. Then $wuv \equiv_{\rm gram} wvu$.
\end{prop}

\begin{proof}
Let $[p,q]$ be a subinterval of $[n]$. We show that $\varphi_n(wvu)_{p,q} = \varphi_n(wuv)_{p,q}$. If $y>q$, then any subsequence of $wvu$ or $wuv$ with letters taken from alphabet $[p,q]$ must, by assumption, be a subsequence of $v$ and so in this case we have $\varphi_n(wvu)_{p,q} = \varphi_n(v) =\varphi_n(wuv)_{p,q}$. If $y<p$, then any subsequence of $wvu$ or $wuv$ with letters taken from alphabet $[p,q]$ must, by assumption, not involve any letter of $v$ and so in this case we have $\varphi_n(wvu)_{p,q} = \varphi_n(wu) =\varphi_n(wuv)_{p,q}$. 

Suppose then that $y \in [p,q]$. By assumption,  the weakly increasing subsequence of $wvu$ formed as $w'u'$, where $w'$ is the subsequence containing all $y$'s in $w$ and $u'$ is a maximal weakly increasing subsequence of $u$ over alphabet $[p,q]$ will have greater length than a $[p,q]$-subsequence of $wvu$ involving a letter of $v$ (which must be of the form $v'u''$ for some weakly increasing subsequences $v'$ of $v$ and $u''$ of $u$ and hence has length bounded above by $|v| + |u'| \leq |w|_y + |u'| = |w'|+|u'|$).  Thus $\varphi_n(wvu)_{p,q}$ is equal to the maximum length of a weakly increasing subsequence of $wvu$ over alphabet $[p,q]$ which does \emph{not} involve any letter from $v$. That is, $\varphi_n(wvu)_{p,q} =\varphi_n(wu)_{p,q}$. Likewise, $w'u'$ (regarded now as a subsequence of $wuv$) also has greater length than any $[p,q]$-subsequence of $wuv$ involving a letter of $v$ (which must be of the form $v'$ for some subsequence $v'$ of $v$ and hence has length bounded above by $|v| \leq |w|_y = |w'| \leq |w'|+|u'|$). Thus $\varphi_n(wvu)_{p,q}$ is also equal to $\varphi_n(wu)_{p,q}$.
\end{proof}
\begin{remark}
Taking $u$ or $v$ to be the empty  word in the previous lemma yields equal words, whilst taking $u$, $v$ and $w$ each to be words of length $1$ yields one of the Knuth relations. Proposition \ref{prop:rel} is not a consequence of the Knuth relations; for example, taking $w=2$, $v=1$ and $u=32$ and applying Schensted's algorithm to the words $wuv=2321$ and $wvu=2132$ does not yield the same result.
\end{remark}
The left patience sorting monoid \cite{CMS} of rank $n$ is the monoid generated by $[n]$ subject to the relations:
$$\mathcal{R}_{{\rm lps}, [n]}:=\{(yu_m\cdots u_1 x,yxu_m\cdots u_1): m \geq 1, x,y, u_i \in [n], x < y \leq u_1 < u_2 <\cdots < u_m \}.$$
\begin{corollary}
The grammic monoid of rank $n$ is a quotient of the left patience sorting monoid of rank $n$.    
\end{corollary}
\begin{proof}
Let $g: [n]^* \rightarrow \mathbb{G}_n$ be the homomorphism sending each $w\in [n]^*$ to its grammic class. It follows from Proposition \ref{prop:rel} that each defining relation of the left patience sorting monoid of rank $n$ holds in the grammic monoid of rank $n$, and so $g$ factors through the left patience sorting monoid. 
\end{proof}
\begin{lemma}
\label{lem:123}
For $n \leq 3$ the grammic congruence on $[n]^*$ is the least congruence generated by 
$\mathcal{R}_n:=\mathcal{R}_{{\rm plac}, [n]} \cup \mathcal{R}_{{\rm lps}, [n]}$.    
\end{lemma}
\begin{proof}
Since every relation in $\mathcal{R}_n$ holds in the grammic monoid, it suffices to show that for any two equivalent words there is a sequence of relations in $\mathcal{R}_n$ that can be applied to obtain one word from the other. For $n=1$, this is trivial, since the grammic monoid is the free monoid of rank $1$, and indeed $\mathcal{R}_1 = \emptyset$ in this case. For $n=2$, the grammic monoid of rank $2$ coincides with the plactic monoid of rank $2$, and here we indeed have $\mathcal{R}_2 = \mathcal{R}_{{\rm plac},[2]}$ since $\mathcal{R}_{{\rm lps},[2]}\subseteq \mathcal{R}_{{\rm plac},2}$. For $n=3$, \cite[Theorem 1]{C} gives that the grammic monoid of rank $3$ is determined by the plactic relations together with $(3212, 2132)$. Noting that:
$3212\equiv_{\rm plac} 2321 \equiv_{\rm lps} 2132,$
we see that this relation lies in $\mathcal{R}_3$, and therefore the relations $\mathcal{R}_3$ determine the grammic monoid of rank $3$. 
\end{proof}

\begin{remark}
\label{rem:Cconj}
For $n=4$ Choffrut conjectured that the congruence is generated by the Knuth relations together with the pairs $(dbac, badc)$ where $a < b \leq c < d$. We note that each such pair lies in $\mathcal{R}_4$ since
$dbac\equiv_{\rm plac} dbca \equiv_{\rm plac} bdca \equiv_{\rm lps} badc\equiv_{\rm plac} bdac.$ 
\end{remark}

\subsection{Column relations}
A different presentation of the plactic monoid is provided in \cite{CGM15}. Let 
$$\mathcal{C}_n = \{w \in [n]^+: P(w) \mbox{ is a single column}\}.$$ Note that elements of $\mathcal{C}_n$ are precisely the words of the form $w_1\cdots w_m$ where each $w_i \in [n]$ and $n \geq w_1 >w_2 > \cdots >w_m \geq 1$. Thus the elements of $\mathcal{C}_n$ are in bijection with the non-empty subsets of $[n]$ and hence $\mathcal{C}_n$ is finite. We shall refer to the elements of $\mathcal{C}_n$ as `column words' on alphabet $[n]$. There is a natural partial order on $\mathcal{C}_n$ via $p \succeq q$ if and only if there exists a tableau in which the column of $p$ appears to the left of the column of $q$. A key observation from \cite{CGM15} is that if $p,q \in \mathcal{C}_n$ and $p \not\succeq q$, then the tableau $P(pq)$ has at most $2$ columns, and moreover $\langle \mathcal{C}_n | \mathcal{T}_{n}\rangle$ is a finite presentation of the plactic monoid of rank $n$ where
\begin{eqnarray*}
\mathcal{T}_{_n} &=& \{(pq,cd): p,q,c,d\in \mathcal{C}_n, p \not\succeq q,  c \succeq d, P(pq) = P(cd)\}\\
&&\cup  \;\;\{(pq,c): p,q,c \in \mathcal{C}_n, p \not\succeq q, P(pq) = P(c)\}.
\end{eqnarray*}
The set of `column readings' of tableaux provides a set of normal forms in the generators $\mathcal{C}_n$.  

The set of column words $\mathcal{C}_n$ is also a generating set for the grammic monoid of rank $n$, although as we have seen, a single grammic class can contain column readings of distinct tableaux. Recalling that whenever $w \equiv_{\rm gram} v$ we have that $P(w)$ and $P(v)$ have the same bottom row and the same content, in order to construct an (infinite) presentation of the grammic monoid, it will therefore suffice to describe for each $m \geq 2$ the pairs of words $(w,v)$ with $w \equiv_{\rm gram} v$ such that $w$ and $v$ are column readings of distinct tableaux with the same bottom row of length $m$. We shall refer to such pairs as `$m$-column grammic relations'. Clearly for each fixed $m$, the number of $m$-column grammic relations is finite. Ultimately, we would like to know whether there is a \emph{finite} set $\mathcal{W}_n$ of such `column relations' giving a finite presentation $\langle \mathcal{C}_n|\mathcal{T}_n\cup \mathcal{W}_n\rangle$ for the grammic monoid.

To begin with, we focus on $2$-column grammic relations. For each column word $c \in \mathcal{C}_n$ and each $1 \leq i \leq n$ such that $P(c)$ has an $i$th row,  let us write $P_i(c)$ to denote the entry in the $i$th row of $P(c)$. Thus, if $P(c)$ has $k$ rows, we have $c = P_k(c)\cdots P_1(c)$. For the rest of this subsection, let:
\begin{eqnarray}
\label{eq:6} c_1, c_2, d_1, d_2 \in \mathcal{C}_n \;  \mbox{  such that }:&& c_1 \succeq c_2 \mbox{  and }d_1 \succeq d_2\\
\label{eq:7}&& P(c_1c_2) \neq P(d_1d_2)\\
\label{eq:8}&&{\rm cont}(c_1c_2) = {\rm cont}(d_1d_2)\\
\label{eq:9}&& P_1(c_1) = P_1(d_1)\leq P_1(c_2) = P_1(d_2). 
\end{eqnarray}
Given \eqref{eq:6}-\eqref{eq:9}, we wish to characterise the circumstances under which $c_1c_2 \equiv_{\rm gram} d_1d_2$.

\begin{lemma}
\label{lem:two-col}
Suppose $c_1, c_2, d_1, d_2$ are as in \eqref{eq:6}-\eqref{eq:9} and that for some $1\leq i \leq n$
\begin{itemize}
    \item[(I)] $P(d_1d_2)$ can be obtained from $P(c_1c_2)$ by moving $i$ from the first column to the second column;
    \item[(II)] $c_2$ contains $i'$, the greatest letter strictly smaller than $i$ occurring in $c_1c_2$;
    \item[(III)] if $c_2$ contains a letter $z$ greater than $i$, then $c_1$ contains $z'$ such that $i<z' \leq z$.
\end{itemize}
Then $c_1c_2 \equiv_{\rm gram} d_1d_2$.
\end{lemma}

\begin{proof}
To see that $c_1c_2 \equiv_{\rm gram} d_1 d_2$, we demonstrate that for all subintervals $[p, q]$ of $[n]$, $\varphi_n(c_1c_2)_{p,q} = \varphi_n(d_1d_2)_{p,q}$. Since these values correspond to the maximum length of a weakly increasing subsequence over alphabet $[p,q]$ in the column products $c_1c_2$ and $d_1d_2$ respectively, we note that each is bounded above by $2$. Moreover, by equation \eqref{eq:8} we know that $\varphi_n(c_1c_2)_{p,q} = 0$ if and only if $\varphi_n(d_1d_2)_{p,q} = 0$. It therefore suffices to show that $\varphi_n(c_1c_2)_{p,q} = 2$ if and only if $\varphi_n(c_1c_2)_{p,q} = 2$.  To this end, first note that if all length two weakly increasing subsequences of $c_1c_2$ over alphabet $[p,q]$ do not contain the $i$ from $c_1$, and all length two weakly increasing subsequences of $d_1d_2$ over alphabet $[p,q]$ do not contain the $i$ from $d_2$, then the equivalence holds immediately (this will be the case if $i \not\in [p,q]$, for example).

Suppose then that $p \leq i \leq z \leq q$ and $i \leq  z$ is a length $2$ weakly increasing subsequence of $c_1c_2$ over alphabet $[p,q]$. By (I) we know that $i$ is contained in $c_1$ but not $c_2$ and, since $c_1$ is a column word, we must then have that $z$ is a letter of $c_2$ with $z>i$. Thus by condition (III) there exists $z'$ in $c_1$ such that $p \leq i<z' \leq z \leq q$, giving that $z' \leq z$ is a length $2$ weakly increasing subsequence of $d_1d_2$ on alphabet $[p,q]$.  

Conversely, suppose that $p \leq z \leq i \leq q$ and $z \leq i$ is a length $2$ weakly increasing subsequence of $d_1d_2$. By (I) we know that $i$ is contained in $d_2$ but not $d_1$ and, since $d_2$ is a column word, we must then have that $z$ is a letter of $d_1$ with $z<i$.  By condition (II), $c_2$ contains $i'$, and so $z \leq i'$ is a length $2$ weakly increasing subsequence of $c_1c_2$. 
\end{proof}

\begin{remark}
\label{rem:Cconj2}
Taking again the case $n=4$, we note that each of the relations suggested by Choffrut is a $2$-column relation from Lemma \ref{lem:two-col}.
\end{remark}
\section{Questions}

We conclude the paper with a number of questions. In terms of descriptions of the grammic monoid, one of the main outstanding questions is:
\begin{question}
Is there a finite presentation for the grammic monoid? 
\end{question}
\noindent
A possible approach to answering this may be further development of the approach via columns used in \cite{CGM15} for the plactic monoid, begun with our Lemma \ref{lem:two-col}. In light of Lemma \ref{lem:123} and Remark \ref{rem:Cconj}, one is also drawn to ask:
\begin{question}
Is the grammic congruence on $[n]^*$ equal to the least congruence generated by $\mathcal{R}_n:=\mathcal{R}_{{\rm plac}, [n]} \cup \mathcal{R}_{{\rm lps}, [n]}$? 
\end{question}

Turning to results concerning varieties of semigroups, Chen, Hu, Luo and O. Sapir \cite{CHLS} have shown that the variety generated by the monoid of $2 \times 2$ upper triangular tropical matrices is not finitely based. Their method does not extend to the monoids of $n \times n$ upper triangular tropical matrices for $n \geq 3$, and the question of whether such monoids are non-finitely based is open. In light of Corollary \ref{cor:identities}, the question can be rephrased as follows:

\begin{question}
Is the semigroup variety generated by the grammic monoid of rank $n$ non-finitely based?
\end{question}

In previous work with Cain and Kambites, we have shown that the variety generated by a single hypoplactic, stalactic or taiga monoid of rank at least $2$ coincides with the variety generated by the natural numbers together with a fixed finite monoid, forming a proper subvariety of the variety generated by the plactic monoid of rank $2$. Moreover, the variety generated by a single finite rank hypoplactic (stalactic, taiga) monoid of rank at least $2$ is equal to the variety generated by the infinite rank hypoplactic (stalactic, taiga) monoid.  By Corollary \ref{cor:varieties} and Corollary \ref{cor:infrank}, we know that the same statements do not hold for the grammic monoid, but one can ask: 
\begin{question}
Does there exist for each $n$ a \emph{finite} monoid $F_n$ such that the variety generated by the grammic monoid of rank $n$ is equal to either the variety generated by $F_n$, or the join of the variety generated by $F_n$ with the variety of commutative semigroups?
\end{question}

Finally, in light of Remark \ref{rem:cyclic}:

\begin{question}
What is the maximum diameter of a connected component of the cyclic shift graph of the grammic monoid of rank $n$?
\end{question}

\section*{Acknowledgements}
The authors were supported through the FCT – Fundação para a Ciência e a Tecnologia, I.P., under the scope of the projects UID/297/2025 and UID/PRR/297/2025 
(Centre for Mathematics and Applications - NOVA Math), which in particular funded a research visit of
the first author to Universidade NOVA Lisboa, where this research was carried out.


\begin{thebibliography}{99}
\bibitem{AR} A. Abram, C. Reutenauer, The stylic monoid, Semigroup Forum 105, 1–-45, 2022.

\bibitem{AR2} A. Abram, C. Reutenauer, F. V. Saliola, Quivers of stylic algebras, Algebr. Comb. 6, 1621–-1635, 2023.

\bibitem{A} T. Aird,  Identities of tropical matrix semigroups and the plactic monoid of rank $4$, Internat. J. Algebra Comput. 32, 1083–-1100, 2022.

\bibitem{AR3} T. Aird, D. Ribeiro,  Tropical representations and identities of the stylic monoid, Semigroup Forum 106, 1–-23, 2023.

\bibitem{AR4} T. Aird, D. Ribeiro, Plactic-like monoids arising from meets and joins of stalactic and taiga congruences, Journal of Algebra 660, 795--851, 2024.

\bibitem{CGM15} A. J. Cain, R. D. Gray, A.  Malheiro, 
Finite Gröbner-Shirshov bases for plactic algebras and biautomatic structures for plactic monoids.
J. Algebra 423 (2015), 37–-53.

\bibitem{CKKMO} A. J. Cain, G. Klein, L. Kubat, A. Malheiro, J. Okni\'{n}ski, A note on the identities in plactic monoids and monoids of upper-triangular tropical matrices, https://arxiv.org/pdf/1705.04596.

\bibitem{CM} A. J. Cain,  A. Malheiro, Combinatorics of cyclic shifts in plactic, hypoplactic, sylvester, Baxter, and related monoids. J. Algebra 535 (2019), 159--224.

\bibitem{CHLS} Y. Chen, X. Hu, Y. Luo, O. Sapir, The finite basis problem for the monoid of two-by-two upper triangular tropical matrices,  Bull. Aust. Math. Soc. 94, 54–-64, 2016.

\bibitem{C} C. Choffrut, Grammic monoids with three generators, Semigroup Forum 105,  680–-692, 2022.

\bibitem{CMS} A. J. Cain, A. Malheiro, F. M.  Silva, The monoids of the patience sorting algorithm,  Internat. J. Algebra Comput. 29,  85–-125, 2019.

\bibitem{DJK} L. Daviaud, M. Johnson, M. Kambites, Identities in upper triangular tropical matrix semigroups and the bicyclic monoid, J. Algebra 501, 503–-525, 2018.

\bibitem{GR} A. Grau, D. Ribeiro, Extended Green’s relations in the Stylic monoid, private communication.

\bibitem{I} Z. Izhakian, Semigroup identities in the monoid of triangular tropical matrices, Semigroup Forum 88, 145–-161, 2014.

\bibitem{I2} Z. Izhakian, Tropical plactic algebra, the cloaktic monoid, and semigroup representations, Journal of Algebra 524 (2019), 290--366.

\bibitem{JF} M. Johnson, P. Fenner, Identities in unitriangular and gossip monoids,  Semigroup Forum 98, 338–-354, 2018.

\bibitem{JK} M. Johnson and M. Kambites, Tropical Representations of Plactic Monoids, Transactions of the American Mathematical Society 374 (2021), no. 6, 4423–-4447.

\bibitem{K} M. Kambites, Free objects in triangular matrix varieties and quiver algebras over semirings, J. Algebra 590, 439–-462, 2022.

\bibitem{Kn} D. E. Knuth, Permutations, matrices, and generalized Young tableaux, Pacific J. Math. 34 (1970), 709–-727.

\bibitem{KO} L. Kubat,J. Okni\'{n}ski, Identities of the plactic monoid, Semigroup Forum 90, 100–112, 2015.

\bibitem{LS} A. Lascoux, M-P. Schützenberger, Le mono\"{i}de plaxique. (French, with Italian summary), Noncommutative structures in algebra and geometric combinatorics (Naples, 1978), Quad. “Ricerca Sci.”, vol. 109, CNR, Rome, 1981, pp. 129–-156.

\bibitem{O} J. Okni\'{n}ski, Identities of the semigroup of upper triangular tropical matrices,  Comm. Algebra 43, 4422–-4426, 2015.

\bibitem{S} C. Schensted, Longest increasing and decreasing subsequences, Canadian J. Math. 13 (1961), 179–-191.

\bibitem{Vs} M. Volkov,  Identities of the stylic monoid,  Semigroup Forum 105, 345-–349, 2022.

\bibitem{V} M. Volkov, Remark on the identities of the grammic monoid with three generators, Semigroup Forum 106 (2023) 332–-337.
\end{thebibliography}
\end{document}